\documentclass{article}
\usepackage{amssymb,amsmath,amsthm,graphicx,cite}

\def\tr{\triangleright}

\def\bar{\overline}

\newtheorem{theorem}{Theorem}

\newtheorem{proposition}[theorem]{Proposition}

\theoremstyle{definition}
\newtheorem{example}{Example}
\newtheorem{definition}{Definition}

\title{\Large \textbf{$\Delta$-Tribrackets and Link Homotopy}}

\date{}

\author{Eric Chavez\footnote{Email: echavez@g.hmc.edu}
\and
Sam Nelson\footnote{Email: Sam.Nelson@cmc.edu. Partially supported by Simons Foundation collaboration grant 702597.}}

\begin{document}
\maketitle

\begin{abstract}
We define a type of Niebrzydowski tribracket we call 
\textit{$\Delta$-tribrackets} and show that their counting invariants are 
invariants of link-homotopy. We further identify several classes of 
tribrackets whose counting invariants for oriented classical knots and links 
are trivial, including vertical tribrackets satisfying the center-involutory 
condition and horizontal tribrackets satisfying the late-commutativity 
condition. We provide examples and end with questions for future research.
\end{abstract}

\parbox{6in} {\textsc{Keywords:} Tribrackets, Link Homotopy, counting 
invariants, ternary algebras

\smallskip

\textsc{2020 MSC:} 57K12}

\section{\large\textbf{Introduction}}\label{I}

A \textit{Niebrzydowski tribracket}, also called a  \textit{knot-theoretic
ternary quasigroup}, is a set with a ternary operation satisfying axioms
encoding the Reidemeister moves in combinatorial knot theory. Specifically,
given a finite tribracket $X$ and an oriented knot or link diagram $D$, the
set of assignments of elements of $X$ to regions in $D$ satisfying a certain
condition at each crossings, known as \textit{$X$-colorings of $D$}, defines 
an invariant of links. See \cite{Nie1} for more.

Associated to any oriented link $L$ is a \textit{fundamental tribracket} 
$\mathcal{T}(L)$ whose
isomorphism class is an invariant of links. From a diagram of an
oriented link we can obtain a presentation of this fundamental tribracket.
The set of $X$-colorings of $D$ can be identified with the set 
$\mathrm{Hom}(\mathcal{T}(L),X)$ of tribracket homomorphisms from the 
fundamental tribracket $\mathcal{T}(L)$ of $L$ to the coloring tribracket 
$X$. Naturally, different diagrams represent the same homomorphism with
different colorings, analogously to how different choices of basis represent
the same linear transformation with different matrices. The cardinality of this
homset is a non-negative integer-valued invariant of oriented knots and links
known as the \textit{tribracket counting invariant}, denoted 
$\Phi_X^{\mathbb{Z}}(L)$.

In this paper we identify a few families of tribrackets whose counting
invariants are trivial on knots, yielding some easily-checkable conditions
on which coloring tribrackets to avoid when using these invariants to
distinguish knots. In particular we define \textit{$\Delta$-tribrackets} whose
counting invariants are trivial on knots and yield link-homotopy invariants
on links with multiple components. The paper is organized
as follows. In Section \ref{TB} we recall the basics of tribracket theory.
In Section \ref{TQ} we identify some families of tribrackets with trivial
invariants arising from known unknotting moves. We define $\Delta$-tribrackets
using the $\Delta$-move, noting that the resulting tribrackets are trivial
for knots and define link-homotopy invariants for links. We identify necessary
and sufficient conditions for Alexander tribrackets to satisfy the conditions
described. In particular the fundamental $\Delta$-tribracket of an oriented 
link has the potential to yield Alexander ideal- or polynomial-style 
invariants of link homotopy type, to be the subject of future investigation.
In Section \ref{E} we provide some example computations and some 
tables of invariant values of the counting invariant for some finite 
$\Delta$-tribrackets on all prime links with up to seven crossings. 
We conclude in Section \ref{Q} with some questions for future research.

\section{\large\textbf{Tribracket Basics}}\label{TB}

We begin with a definition; see \cite{Nie1,NOO} for more.

\begin{definition}
Let $X$ be a set. A ternary map $[,,]:X\times X\times X \to X$ defines
\textit{horizontal tribracket} structure on $X$ if it satisfies the conditions
\begin{itemize}
\item[(i)] For all ordered triples $(x,y,z)\in X\times X\times X$, there exist
unique elements $u,v,w\in X$ such that
\[[x,y,u]=z, \quad [x,v,y]=z\quad \mathrm{and} \quad [w,x,y]=z\]
and
\item[(iih)] For all ordered quadruples $(x,y,z,w)\in X\times X\times X\times X$
we have
\[[y,[x,y,z],[x,y,w]]=[z,[x,y,z],[x,z,w]]=[w,[x,y,w],[x,z,w]].\]
\end{itemize}
A ternary map $\langle,\rangle:X\times X\times X\to X$ defines a 
\textit{vertical tribracket} structure on $X$ if it satisfies the conditions
\begin{itemize}
\item[(i)] For all ordered triples $(x,y,z)\in X\times X\times X$, there exist
unique elements $u,v,w\in X$ such that
\[\langle x,y,u\rangle=z, \quad \langle x,v,y\rangle =z\quad \mathrm{and} 
\quad \langle w,x,y\rangle=z\]
and
\item[(iiv)] For all ordered quadruples $(x,y,z,w)\in X\times X\times X\times$ 
we have
\begin{eqnarray*}
\langle a,\langle x,y,z\rangle,\langle \langle x,y,z\rangle,z,w \rangle \rangle
& = & \langle x,y,\langle y,z,w\rangle \rangle \\
\langle \langle x,y,z\rangle,z,w \rangle & = & 
\langle \langle x,y,\langle y,z,w\rangle \rangle,\langle y,z,w\rangle,w \rangle.
\end{eqnarray*}
\end{itemize}
\end{definition}

Tribracket axiom (i) says that the tribracket operation has left-, center- and 
right-invertibility. 

Tribrackets form a category whose objects are tribrackets and whose morphisms
are \textit{tribracket homomorphisms}, i.e functions $f:X\to Y$ satisfying
for all $x,y,z\in X$
\[f([x,y,z])=[f(x),f(y),f(z)]\]
for vertical tribrackets and
\[f(\langle x,y,z\rangle )=\langle f(x),f(y),f(z)\rangle\]
for horizontal tribrackets.

The two types of tribrackets are equivalent in the sense that for any
vertical tribracket structure on $X$ there is a corresponding horizontal
tribracket structure on $X$ and vice-versa satisfying for all $x,y,z\in X$
\[z=\langle x,y,[x,y,z]\rangle =[x,y,\langle x,y,z\rangle]. \]
Note that we use square brackets to indicate horizontal tribracket operations 
and angle brackets to indicate vertical tribracket operations without further 
comment.

\begin{example}
Let $G$ be a group; then the ternary operation $[x,y,z]=yx^{-1}z$ defines a
horizontal tribracket known as the \textit{Dehn tribracket} of the group.
\end{example}

\begin{example}
Let $M$ be a module over the ring to two-variable Laurent polynomials
with integer coefficients $\mathbb{Z}[t^{\pm 1},s^{\pm 1}]$. Then $M$ is a 
horizontal tribracket with ternary operation $[x,y,z]=ty+sz-tsx$ known
as an \textit{Alexander tribracket}.
\end{example}

\begin{example}
We can specify a tribracket structure on a finite set $X=\{1,\dots, n\}$ 
with an \textit{operation 3-tensor}, i.e. an ordered
$n$-tuple of $n\times n$ matrices with entries in $\{1,2,\dots, n\}$
such that the entry in matrix $i$, row $j$ column $k$ is $[i,j,k]$.
For instance, there are two tribrackets of order 2:
\[\left[\left[\begin{array}{rr} 1 & 2 \\ 2 & 1\end{array}\right],
\left[\begin{array}{rr} 2 & 1 \\ 1 & 2\end{array}\right]\right]\quad
\mathrm{and}\quad
\left[\left[\begin{array}{rr} 2 & 1 \\ 1 & 2\end{array}\right],
\left[\begin{array}{rr} 1 & 2 \\ 2 & 1\end{array}\right]\right].\]
\end{example}

\begin{example}
Let $G=\{x_1,\dots, x_n\}$ be a set of generators. A \textit{tribracket
word} in $G$ is either an element of $G$ or a string of symbols of one of 
the following forms:
\[\{[x,y,z],\ [\bar{x},y,z],\ [x,\bar{y},z], [x,y,\bar{z}]\}\]
where $x,y,z\in G$.
Then a \textit{tribracket presentation} is a finite list $G$ of generators
and \textit{relations}, i.e., equations of tribracket words in $G$. The
presented tribracket is the set of equivalence classes on tribracket words
under the equivalence relation generated by the equivalences
\[\begin{array}{rcl}
x & \sim & [[\bar{x},y,z],y,z] \\
y & \sim & [x,[x,\bar{y},z],z] \\
z & \sim & [x,y,[x,y,\bar{z}]] \\
{}[y,[x,y,z],[x,y,w]] & \sim & [z,[x,y,z],[x,z,w]] \\
{}[z,[x,y,z],[x,z,w]] & \sim & [w,[x,y,w],[x,z,w]]
\end{array}
\]
for all tribracket words $x,y,z,w$ together with the explicitly listed 
relations. 

In particular, the words $[\bar{x},y,z],\ [x,\bar{y},z]$ and 
$[x,y,\bar{z}]$ are the left-, center- and right-inverses required
by the first tribracket axiom; note that $\bar{x}$ by itself is not a
tribracket word.
\end{example}

\begin{definition}
Let $L$ be an oriented link diagram. The \textit{fundamental tribracket}
of $L$, $\mathcal{T}(L)$ has a presentation with a generator for each region 
in the planar 
complement of $L$ and a relation at each crossing of the form
\[\includegraphics{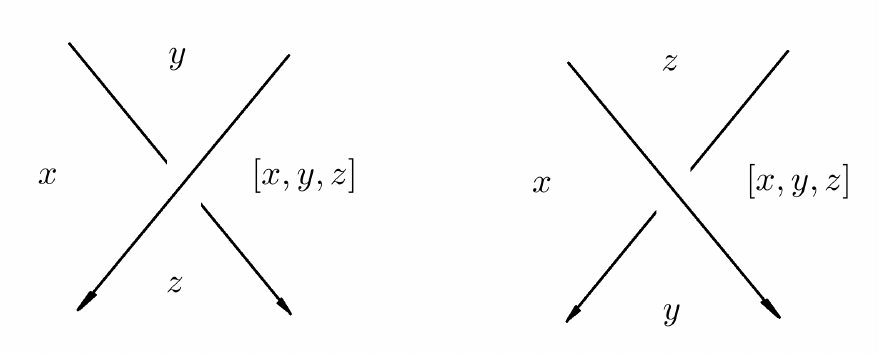}\]
\end{definition}

\begin{example}
The right-handed trefoil knot has fundamental tribracket
with presentation
\[\mathcal{T}(3_1)=\langle x,y,z,w \ |  [x,y,z]=[x,z,w]=[x,w,y]\rangle.\]
\[\includegraphics{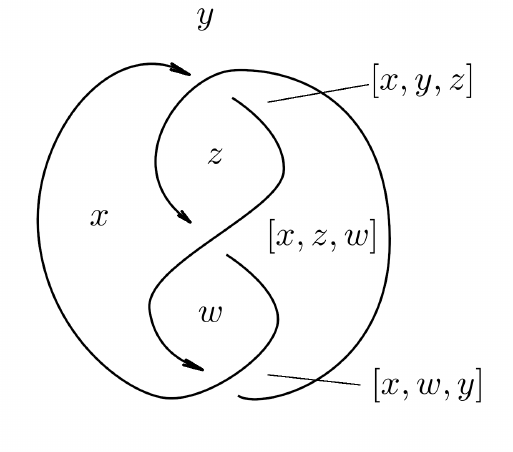}\]
\end{example}

Reidemeister moves on $L$ induce Tietze moves on $\mathcal{T}(L)$; hence it 
follows that $\mathcal{T}(L)$ is an invariant of oriented knots and links.
Since comparing isomorphism classes of objects presented by generators and
relations is generally inconvenient, we can use $\mathcal{T}(L)$ to define 
more computable invariants. One strategy for doing so involves computing
sets of tribracket homomorphisms, i.e. maps $f:\mathcal{T}(L)\to X$
from the fundamental tribracket of $L$ to a finite tribracket $X$
satisfying the condition
\[f([x,y,z])=[f(x),f(y),f(z)]\]
at every crossing. Such homomorphisms can be conveniently represented as
\textit{region colorings} of a diagram of $L$ by elements of $X$; see
\cite{ANR,NNS,NP1,NP2} etc.

\section{\large\textbf{Trivializing Quotients and Trivial Invariants}}\label{TQ}

For other knot-theoretic algebraic structures such as quandles and biquandles,
quotients of the fundamental algebraic structure associated to a knot
or link can provide interesting invariants. For example, the question
of which knots and links have finite \textit{fundamental kei} (also called
\textit{fundamental involutory quandle}), obtained from the fundamental
quandle of the knot or link by including relations $(x\tr y)\tr y=x$ for all
$x,y$, was considered in \cite{W} and later in \cite{HS}. Similarly, involutory 
quotients of knot biquandles were studied in \cite{CN} and quotients of 
fundamental virtual quandles of virtual knots were studied in \cite{NT}.
Any homomorphism from a knot quandle to an involutory target quandle must
factor through the fundamental involutory quandle of the knot. We will now
apply this principle to identify some families of tribrackets whose counting
invariants are trivial.

\begin{definition}
Say that a horizontal tribracket $X$ is \textit{left-}, \textit{center-} or 
\textit{right-involutory} respectively if the functions defined by
\[f_{x,y}(a)=[a,x,y],\quad g_{x,y}(a)=[x,a,y]\quad 
\mathrm{and}\quad h_{x,y}(a)=[x,y,a]\]
respectively are involutions for all $x,y\in X$, i.e., if
\[f_{x,y}(f_{x,y}(a))=g_{x,y}(g_{x,y}(a))=h_{x,y}(h_{x,y}(a))=a.\] Similarly a 
vertical
tribracket is left-, center- or right-involutory if for all $x,y\in X$ the 
functions
\[f_{x,y}(a)=\langle a,x,y\rangle ,\quad g_{x,y}(a)=\langle x,a,y\rangle
\quad \mathrm{and}\quad h_{x,y}(a)=\langle x,y,a\rangle\]
respectively are involutions. Say that a tribracket is \textit{fully involutory}
if it is involutory in all three positions.
\end{definition}

Let $X$ be a center-involutory vertical tribracket. We then note that 
the number of colorings of any oriented knot or link by $X$ is unchanged by
$2$-moves:
\[\includegraphics{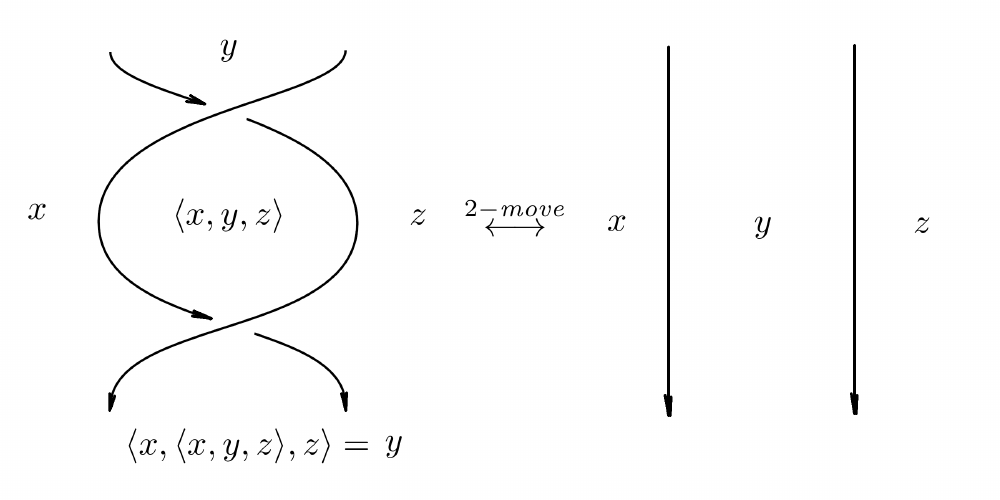}\]
Since the 2-move followed by a Reidemeister II move amounts to a crossing 
change, the fundamental central-involutory vertical tribracket of an
oriented knot or link does not change under crossing change, and 
we have:

\begin{proposition}\label{prop:inv}
The number of colorings of any oriented link diagram by a center-involutory
vertical tribracket is the same as that of the unlink with the same
number $c$ of components, namely $\Phi_X^{\mathbb{Z}}(L)=|X|^{c+1}$.
\end{proposition}

\begin{proof}
Let $L$ be an oriented link and consider the fundamental tribracket of
$L$. Adding the relation $\langle x,\langle x,y,z\rangle, z\rangle =y$ 
for all $x,y,z$ yields the fundamental center-involutory tribracket of $L$, 
denoted $\mathcal{T}_{LC}(L)$. The center-involutory condition implies that 
the presented tribracket is unchanged by crossing-change moves.
Hence, the fundamental center-involutory tribracket is invariant
under crossing changes; in particular, it is isomorphic to that 
of the unlink $U_c$ with same number $c$ of components as $L$.

If $X$ is an center-involutory tribracket, then any homomorphism
$f:\mathcal{T}(L)\to X$ must factor through $\mathcal{T}_{LC}(L)$. Since
each homomorphism $f:\mathcal{T}(L)\to \mathcal{T}_{LC}(L)$ is just taking
a quotient, it follows that 
$|\mathrm{Hom}(\mathcal{T}(L),X)|=|\mathrm{Hom}(\mathcal{T}_{LC}(L),X)|$.
Then since   
$|\mathrm{Hom}(\mathcal{T}_{LC}(L),X)|=|\mathrm{Hom}(\mathcal{T}_{LC}(U_c),X)|
=|X|^{c+1}$, we have the result.
\end{proof}

Next we have another category of tribrackets with trivial counting invariants.

\begin{definition}
Let $X$ be a horizontal tribracket. We say $X$ is \textit{late-commutative}
if for all $x,y,z\in X$ we have
\[[x,y,z]=[x,z,y].\]
\end{definition}

\begin{example}
An Alexander tribracket is late-commutative if $t=s$, i.e. 
\[[x,y,z]=ty+tz-t^2x.\]
\end{example}

\begin{example}
A Dehn tribracket $G$ is late-commutative if for all $a,b,c\in G$ we have
\[ab^{-1}c=cb^{-1}a.\]
Then in particular, we must have $ac=ca$ when $b$ is the identity, and 
$G$ is abelian.
\end{example}

We then have
\begin{proposition}
If $X$ is a late-commutative tribracket and $K$ is an oriented link, then
$\Phi_X^{\mathbb{Z}}(K)=|X|^{c+1}$ where $c$ is the number of components of $L$.
\end{proposition}

\begin{proof}
This is similar to the proof of Proposition \ref{prop:inv}.
\end{proof}

A more subtle category of tribrackets with coloring invariants which are trivial
on knots but provide link-homotopy invariants for links
is what we call \textit{$\Delta$-tribrackets:}
\begin{definition}
Let $X$ be a tribracket. We say $X$ is a \textit{$\Delta$-tribracket}
if for all $x,y,z,w\in X$ we have
\[[y,[x,y,z],[x,w,y]]=[z,[x,z,w],[x,y,z]]=[w,[x,w,y],[x,z,w]].\]
\end{definition}

\begin{example}
An Alexander tribracket is a \textit{$\Delta$-tribracket} if it satisfies
the property $2st=t^2+s^2$, since
\begin{eqnarray*}
{}[y,[x,y,z],[x,w,y]] 
& = & -tsy +t(-tsx+ty+sz)+s(-tsx+tw+sy) \\
& = & (t^2+s^2-ts)y -(t+s)ts x +tsw+tsz, \\
{}[z,[x,z,w],[x,y,z]] & = & (t^2+s^2-ts)z -(t+s)ts x +tsy+tsw\ \mathrm{and} \\
{}[w,[x,w,y],[x,z,w]] & = & (t^2+s^2-ts)w -(t+s)ts x +tsy+tsz 
\end{eqnarray*}
so the condition we need is 
\[t^2+s^2-st=st.\]
Then for example, $\mathbb{Z}_8$ is a $\Delta$-tribracket with $t=3$ and $s=7$,
as is $\mathbb{Z}_9$ with $t=2$ and $s=5$.
\end{example}

The name ``$\Delta$-tribracket'' is chosen to reflect that colorings
by a $\Delta$-tribracket are unchanged $\Delta$-moves:
\[\includegraphics{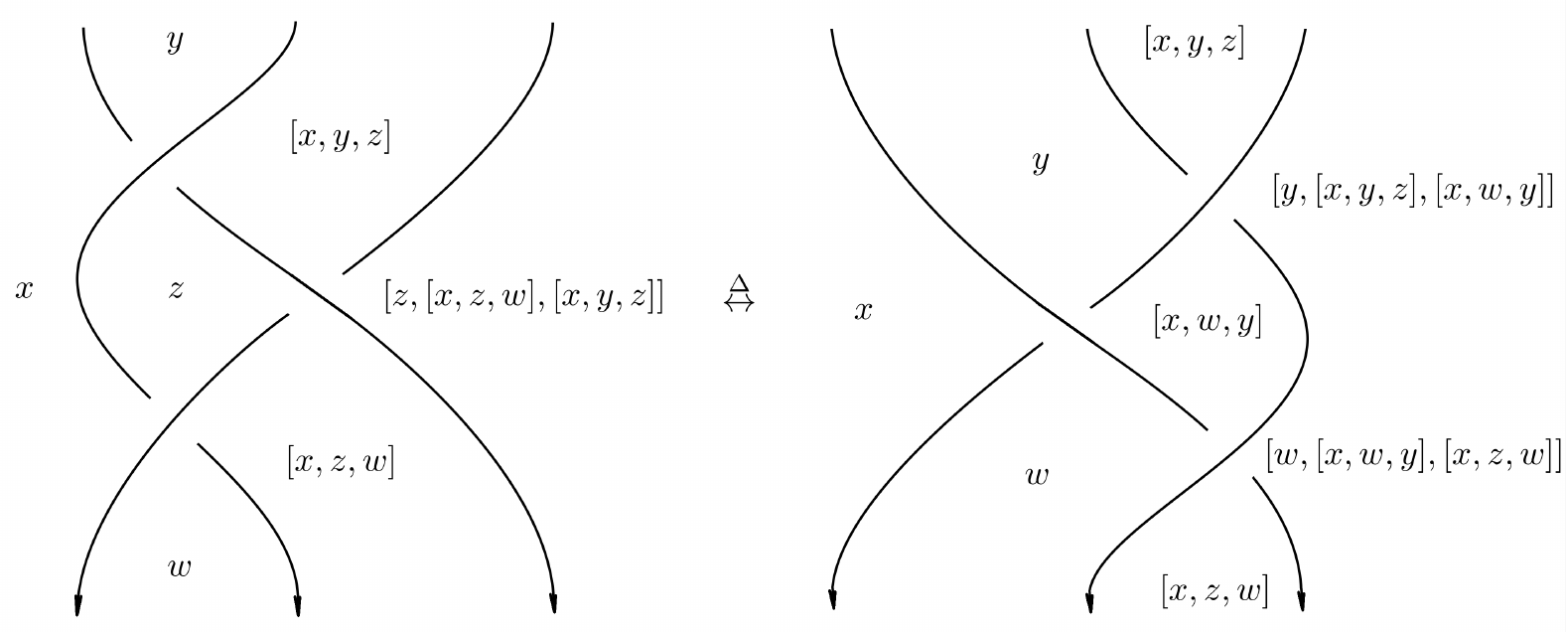}\]

Recall from \cite{MN} that
\begin{itemize}
\item The $\Delta$-move is an unknotting move, and
\item Links of several components are link-homotopic iff they are
related by $\Delta$-moves together with Reidemeister moves.
\end{itemize} 

As a consequence we have:
\begin{proposition}
The counting invariant $\phi_X^{\mathbb{Z}}$ associated to a $\Delta$-tribracket
is a link-homotopy invariant.
\end{proposition}

\section{\large\textbf{Examples}}\label{E}

In this section we collect a few examples.

\begin{example}\label{ex:trefoil1}
Consider the trefoil knot $3_1$ below.
\[\includegraphics{ec-sn-2.pdf}\]
As we have seen, its fundamental tribracket has presentation
\[\mathcal{T}(3_1)=\langle x,y,z,w\ | [x,y,z]=[x,z,w]=[x,w,y] \rangle.\]
As an Alexander tribracket, this has presentation matrix
\[
\left[\begin{array}{rrrrr}
1 & st & -t & -s & 0  \\
1 & st & 0 & -t & -s  \\
1 & st & -s & 0 & -t  \\
\end{array}\right].\]
In the case of the fundamental late-commutative tribracket this is
\[\left[\begin{array}{rrrrr}
1 & t^2 & -t & -t & 0  \\
1 & t^2 & 0 & -t & -t  \\
1 & t^2 & -t & 0 & -t  
\end{array}\right]\]
which row-reduces over $\mathbb{Z}[t^{\pm 1}]$ to
\[\left[\begin{array}{rrrrr}
t^{-1} & t & 0 & 0 & -2  \\
0 & 0   & 1 & 0 & -1  \\
0 & 0   & 0 & 1 & -1  \\
\end{array}\right]\]
with kernel of dimension 2 (the same as the unknot) as expected.
\end{example}

\begin{example} \label{ex:trefoil2}
Continuing with the trefoil from Example \ref{ex:trefoil1}, first
row-reducing the presentation matrix over $\mathbb{Z}[t^{\pm 1},s^{\pm 1}]$
we have
\[
\left[\begin{array}{rrrrr}
1 & st & -t & -s & 0  \\
1 & st & 0 & -t & -s  \\
1 & st & -s & 0 & -t  \\
\end{array}\right]
\rightarrow
\left[\begin{array}{rrrrr}
1 & st & -t & -s & 0  \\
0 & 0 & t & s-t & -s  \\
0 & 0 & t-s & s & -t  \\
\end{array}\right]\]\[
\rightarrow
\left[\begin{array}{rrrrr}
1 & st & -t & -s & 0  \\
0 & 0 & t & s-t & -s  \\
0 & 0 & -s & t & s-t  \\
\end{array}\right]
\rightarrow
\left[\begin{array}{rrrrr}
1 & st & -t & -s & 0  \\
0 & 0 & st & s^2-ts & -s^2  \\
0 & 0 & -st & t^2 & st-t^2  \\
\end{array}\right]\]\[
\rightarrow
\left[\begin{array}{ccccc}
1 & st & -t & -s & 0  \\
0 & 0 & st & s^2-st & -s^2  \\
0 & 0 & 0 & s^2+t^2-st & st-s^2-t^2  \\
\end{array}\right]
\]
So in the $\Delta$-tribracket case, with $t^2+s^2-st=st$, we have
\[
\left[\begin{array}{rrrrr}
1 & st & -t & -s & 0  \\
0 & 0 & st & s^2-st & -s^2  \\
0 & 0 & 0 & st & -st  \\
\end{array}\right]
\rightarrow 
\left[\begin{array}{rrrrr}
1 & st & -t & -s & 0  \\
0 & 0 & st & s^2-st & -s^2  \\
0 & 0 & 0 & 1 & -1  \\
\end{array}\right]
\rightarrow
\left[\begin{array}{rrrrr}
1 & st & -t & -s & 0  \\
0 & 0 & 1 & st^{-1}-1 & -st^{-1}  \\
0 & 0 & 0 & 1 & -1  \\
\end{array}\right]
\]
\[
\rightarrow
\left[\begin{array}{rrrrr}
1 & st & -t & -s & 0  \\
0 & 0 & 1 & 0 & -2st^{-1}+1  \\
0 & 0 & 0 & 1 & -1  \\
\end{array}\right]
\rightarrow
\left[\begin{array}{rrrrr}
1 & st & 0 & 0 & -s+t  \\
0 & 0 & 1 & 0 & -2st^{-1}+1  \\
0 & 0 & 0 & 1 & -1  \\
\end{array}\right]
\]
again with kernel of dimension 2 as expected.
\end{example}

\begin{example}
Taking the case of the Hopf link $L2a1$ below,
\[\includegraphics{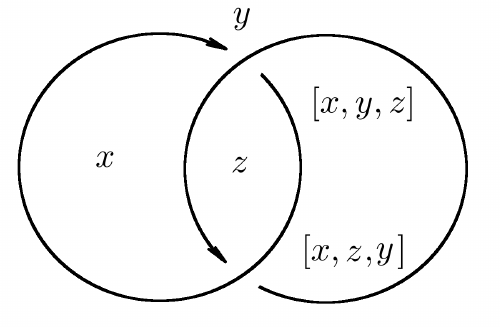}\]
we have fundamental Alexander tribracket matrix presented by
\[
\left[\begin{array}{rrrr}
1 & st & -t & -s \\
1 & st & -s & -t \\
\end{array}\right]
\]
which row-reduces over $\mathbb{Z}$ to
\[
\left[\begin{array}{rrrr}
1 & st & -t & -s \\
0 & 0 & t-s & s-t \\
\end{array}\right].
\]
In the case of the late-commutative tribracket we get kernel of dimension
3 as anticipated, the same as the unlink of two components. More generally, 
for coloring tribrackets in which $t-s$ is a unit, we will have kernel of 
dimension 2, distinguishing the link-homotopy type of the Hopf link from 
that of the unlink. For coloring tribrackets in which $t-s$ is a zero 
divisor, the torsion part of the coloring space can also distinguish the 
link-homotopy type of this link.
\end{example}

\begin{example}\label{ex:delta}
Using \texttt{python} code, we computed the numbers of colorings for
the prime links with up to seven crossings as listed at the Knot Atlas 
(\texttt{www.katlas.org}) with respect the two Alexander $\Delta$-tribrackets 
$X_1=(\mathbb{Z}_8,s=7,t=3)$ and $X_1=(\mathbb{Z}_9,s=5,t=2)$; the results
are collected in the table.
\[
\begin{array}{r|l}
L & \Phi_{X_1}^{\mathbb{Z}}(L) \\ \hline
256 & L2a1, L6a2, L6a3, L7a5, L7a6 \\
512 & L4a1, L5a1, L6a1, L7a1, L7a2, L7a3, L7a4,L7n1, L7n2\\
1024 & L6a1, L6n1, L7a7 \\
4096 & L6a4
\end{array}
\]\[
\begin{array}{r|l}
L & \Phi_{X_2}^{\mathbb{Z}}(L) \\ \hline
243 & L2a1, L4a1, L6a1, L7a2, L7a5, L7a6, L7n1 \\
729 & L5a1, L6a2, L6a3, L7a1, L7a3, L7a4, L7a7, L7n2\\
2187 & L6a5, L6n1 \\
6561 & L6a4
\end{array}
\]
\end{example}

These $\Delta$-tribracket invariants are fairly quick and easily computable, 
providing a useful tool for distinguishing link-homotopy classes. Example 
\ref{ex:delta} shows that non-isomorphic finite $\Delta$-tribrackets provide
different information about link-homotopy type in general. As with all such
invariants, the invariant defined by any given finite $\Delta$-tribracket
is not expected to be a complete invariant of link-homotopy. However, the 
set of all $\Delta$-tribracket invariants could potentially be a complete
invariant.

\section{\large\textbf{Questions}}\label{Q}

We end with a few questions for future research.

What is the exact relationship between the fundamental $\Delta$-tribracket
of a link, its homset invariants, and other link-homotopy invariants such as 
Milnor invariants or the quasi-trivial quandles and biquandle invariants 
in \cite{ELN,I}?

The fundamental Alexander $\Delta$-tribracket of a link seems to potentially
play the role of an Alexander invariant for link-homotopy; what ideal-based
or polynomial invariants can be extracted from it?

What enhancements of the fundamental $\Delta$-tribracket of a link can be 
defined, and what information about link-homotopy type do they extract?

Is the set of all $\Delta$-tribracket counting invariants a complete invariant
of link homomotpy? That is, given any two non-link-homotopic links $L$ and $L'$,
is there always a finite $\Delta$-tribracket $X$ such that
\[\Phi_X^{\mathbb{Z}}(L)\ne\Phi_X^{\mathbb{Z}}(L')?\]
If so, how can we find the smallest such $X$? If not, identify under what
conditions $L$ and $L'$ are not link-homotopic but have isomorphic 
fundamental $\Delta$-tribrackets.

\bibliography{ec-sn}{}

\begin{thebibliography}{10}

\bibitem{ANR}
L.~Aggarwal, S.~Nelson, and P.~Rivera.
\newblock Quantum enhancements via tribracket brackets.
\newblock {\em Mediterr. J. Math.}, 18(1):Paper No. 10, 13, 2021.

\bibitem{CN}
J.~Chien and S.~Nelson.
\newblock Virtual links with finite medial bikei.
\newblock {\em J. Symbolic Comput.}, 92:211--221, 2019.

\bibitem{ELN}
M.~Elhamdadi, M.~Liu, and S.~Nelson.
\newblock Quasi-trivial quandles and biquandles, cocycle enhancements and
  link-homotopy of pretzel links.
\newblock {\em J. Knot Theory Ramifications}, 27(11):1843007, 16, 2018.

\bibitem{HS}
J.~Hoste and P.~D. Shanahan.
\newblock Involutory quandles of {$(2,2,r)$}-{M}ontesinos links.
\newblock {\em J. Knot Theory Ramifications}, 26(3):1741003, 19, 2017.

\bibitem{I}
A.~Inoue.
\newblock Quasi-triviality of quandles for link-homotopy.
\newblock {\em J. Knot Theory Ramifications}, 22(6):1350026, 10, 2013.

\bibitem{MN}
H.~Murakami and Y.~Nakanishi.
\newblock On a certain move generating link-homology.
\newblock {\em Math. Ann.}, 284(1):75--89, 1989.

\bibitem{NNS}
D.~Needell, S.~Nelson, and Y.~Shi.
\newblock Tribracket modules.
\newblock {\em Internat. J. Math.}, 31(4):2050028, 13, 2020.

\bibitem{NOO}
S.~Nelson, N.~Oyamaguchi, and R.~Sazdanovic.
\newblock Psyquandles, singular knots and pseudoknots.
\newblock {\em Tokyo J. Math.}, 42(2):405--429, 2019.

\bibitem{NP1}
S.~Nelson and E.~Pauletich.
\newblock Multi-tribrackets.
\newblock {\em J. Knot Theory Ramifications}, 28(12):1950075, 16, 2019.

\bibitem{NP2}
S.~Nelson and S.~Pico.
\newblock Virtual tribrackets.
\newblock {\em J. Knot Theory Ramifications}, 28(4):1950026, 12, 2019.

\bibitem{NT}
S.~Nelson and S.~Tamagawa.
\newblock Quotient quandles and the fundamental {L}atin {A}lexander quandle.
\newblock {\em New York J. Math.}, 22:251--263, 2016.

\bibitem{Nie1}
M.~Niebrzydowski, A.~Pilitowska, and A.~Zamojska-Dzienio.
\newblock Knot-theoretic ternary groups.
\newblock {\em Fund. Math.}, 247(3):299--320, 2019.

\bibitem{W}
S.~K. Winker.
\newblock {\em Q{UANDLES}, {KNOT} {INVARIANTS}, {AND} {THE} {N}-{FOLD}
  {BRANCHED} {COVER}}.
\newblock ProQuest LLC, Ann Arbor, MI, 1984.
\newblock Thesis (Ph.D.)--University of Illinois at Chicago.

\end{thebibliography}
\bibliographystyle{abbrv}

\bigskip

\noindent
\textsc{Department of Mathematical Sciences \\
Claremont McKenna College \\
850 Columbia Ave. \\
Claremont, CA 91711}

\end{document}